\documentclass{article}

\usepackage[utf8]{inputenc}
\usepackage{csquotes}
\usepackage[english]{babel}

\usepackage{amssymb,amsmath,amsthm,amsfonts}

\usepackage{authblk}
\usepackage{fullpage}
\usepackage[titletoc,toc,title]{appendix}

\usepackage[square,sort,comma,numbers]{natbib}

\usepackage{pgf}
\usepackage{tikz}
\usepackage{float}

\usepackage{titling}
\usepackage{caption}
\usepackage{booktabs}
\usepackage{verbatim}
\usepackage{mathtools}

\usepackage{tikz}
\usepackage{pgfplots}

\pgfplotsset{compat=newest}
\usetikzlibrary{shapes.geometric,arrows,fit,matrix,positioning}
\tikzset{%
    treenode/.style = {draw=none, fill=none, align=center, minimum size=.05cm},
    subtree/.style = {
      isosceles triangle,
      draw=black,
      align=center,
      minimum height=0.5cm,
      minimum width=1cm,
      shape border rotate=90,
      anchor=north}
}
\graphicspath{{img/}}
\DeclareGraphicsExtensions{.png}

\newtheorem{theorem}{Theorem}[section]

\newtheorem{lemma}[theorem]{Lemma}

\newtheorem{definition}[theorem]{Definition}
\newtheorem{example}[theorem]{Example}

\newcommand\len[1]{\left|#1\right|}

\begin{document}

\title{A Self-Referential Property of Zimin Words}
\author{John Connor\\The Graduate Center, City University of New York}

\date{}

\maketitle
\thispagestyle{empty}

\begin{abstract}
    This paper gives a short overview of Zimin words, and proves an
    interesting property of their distribution.
    Let $L_q^m$ to be the lexically ordered sequence of $q$-ary words of
    length $m$,
    and let $T_n(L_q^m)$ to be the binary sequence where the $i$-th term is
    $1$ if and only if the $i$-th word of $L_q^m$ encounters the $n$-th
    Zimin word, $Z_n$.
    We show that the sequence $T_n(L_q^m)$ is an instance of
    $Z_{n+1}$ when $1 < n$ and $m=2^n-1$.
\end{abstract}

\section{Introduction}
In this section we will introduce some basic definitions and give an
informal overview of our research.
The section ends with an informal statement of our main result,
while the rest of the paper develops the necessary framework to formalize
and prove the theorem.
We conclude with a few conjectures and open problems.

Throughout this paper we use $q, n, m$ to denote natural numbers with $1 < q$,
and we define the alphabet $\Sigma$ to be the set of the first $q$ natural numbers.
This gives us a certain flexibility in treating words as numbers represented in base $q$,
but all of our results can of course be made completely general.
For convenience, we define $Q$ to be the word of length $1$ consisting of the symbol $q-1$.

\begin{definition}[Zimin Words]
    We define the Zimin Words as an infinite set of finite words recursively
    constructed over the natural numbers
    \begin{equation}
        \begin{aligned}
            Z_0     &= 0 \\
            Z_{n+1} &= Z_{n} (n + 1) Z_{n}
        \end{aligned}
    \end{equation}

    \begin{example}
        The first four Zimin words are
        \begin{align*}
            Z_0 &= 0 \\
            Z_1 &= Z_0 1 Z_0 = 010 \\
            Z_2 &= Z_1 2 Z_1 = 0102010 \\
            Z_3 &= Z_2 3 Z_2 = 010201030102010
        \end{align*}
    \end{example}
\end{definition}

\begin{definition}[Instance]
    A word $W$ over $\Sigma$ is an instance of a word $V$ over $\Sigma'$ if
    and only if there is a homomorphism $\phi \in \hom(\Sigma', \Sigma^+)$
    such that $\phi(V) = W$.  Note that the codomain of $\phi$ does not
    contain the empty word; morphisms with this constraint are sometimes
    called ``non-erasing morphisms''.
    \begin{example}
        For example if $W=ABCD$ and $V=12$ then $W$ is a $V$-instance under the homomorphism
        \begin{align*}
            1 \mapsto AB \\
            2 \mapsto CD
        \end{align*}
    \end{example}
\end{definition}

\begin{definition}[Subword Relation]
    A word $V$ is a subword of $W$ if and only if $V$ appears contiguously in $W$.
    That is $V \leq W \iff W = XVY$, and $V$, $X$, $Y$ may all be the empty word.
\end{definition}

\begin{definition}[Encounter]
A word $W$ encounters $V$ provided some subword $U \leq W$ is an instance of $V$.
If a word $W$ fails to encounter $V$ then $W$ is said to avoid $V$.
\end{definition}

\begin{definition}[Unavoidable]
    A word $W$ is unavoidable if for any finite alphabet there are only
    finitely many words which avoid $W$ over the alphabet.
\end{definition}

Zimin~\cite{zimin1984} gives the following characterization of unavoidable words
\begin{theorem}
    A word $W$ with $n$ distinct letters is unavoidable if and only if
    $Z_n$ encounters $W$.
\end{theorem}
\noindent
The paper is in Russian,
but there is a proof of this theorem in~\cite{lothaire2002algebraic}
and in~\cite{sapir2014}.

It can be illustrative to label the vertices of a $q$-ary tree with
$q$-ary words, such that the root is labeled with the empty string,
and the word labeling the vertex connected to its parent by the $i$-th
branch is the label of the parent vertex with the symbol $i$ prepended.

The binary tree in Figure~\ref{fig:bintree:0} is constructed by pruning all
vertices which are children of a vertex labeled with a word encountering
$Z_n$.
From the figure, we can see that when $q = 2$,
every word of length less than $3$ avoids $Z_2$,
while there exists only two words of length $4$ which avoid $Z_2$.
\begin{figure}[!ht]
\centering
\begin{tikzpicture}[->,>=stealth', level/.style={sibling distance = 8cm/#1, level distance = 1.5cm}, scale=0.6,transform shape]
    \node [treenode] {$\epsilon$}
    child [sibling distance = 10cm] {%
        node [treenode] {$0$}
        child {%
            node [treenode] {$00$}
            child {node [treenode] {$000$}}
            child {%
                node [treenode] {$100$}
                child {node [treenode] {$0100$}}
                child {%
                    node [treenode] {$1100$}
                    child {node [treenode] {$01100$}}
                    child {node [treenode] {$11100$}}
                }
            }
        }
        child {%
            node [treenode] {$10$}
            child {node [treenode] {$010$}}
            child {%
                node [treenode] {$110$}
                child {node [treenode] {$0110$}}
                child {node [treenode] {$1110$}}
            }
        }
    }
    child [sibling distance = 10cm] {%
        node [treenode] {$1$}
        child {%
            node [treenode] {$01$}
            child {%
                node [treenode] {$001$}
                child {node [treenode] {$0001$}}
                child {node [treenode] {$1001$}}
            }
            child {node [treenode] {$101$}}
        }
        child {%
            node [treenode] {$11$}
            child {%
                node [treenode] {$011$}
                child {%
                    node [treenode] {$0011$}
                    child {node [treenode] {$00011$}}
                    child {node [treenode] {$10011$}}
                }
                child {node [treenode] {$1011$}}
            }
            child {node [treenode] {$111$}}
        }
    }
;
\end{tikzpicture}
\caption{}\label{fig:bintree:0}
\end{figure}

In the next section we will consider general homomorphisms, but for now we
will only concern ourselves with homomorphisms such that
$\left|\phi(i)\right| = 1$ for all $i$.
When $q=2$, we can easily see that there are only $2$ words which encounter
$Z_1$: $0$ and $1$.
We can also easily see that all of the $2$-ary $Z_2$ instances are of the
form $\phi(0)\phi(1)\phi(0)$.
In general we will see that it is always the case that a $Z_{n+1}$ instance
is of the form $\phi(Z_{n})\phi(n+1)\phi(Z_{n})$.
We say that $\phi(Z_n)$ generates $\phi(Z_{n})\phi(n+1)\phi(Z_{n})$,
and that $\phi(Z_{n})\phi(n+1)\phi(Z_{n})$ is generated by $\phi(Z_n)$.

Now consider all of the $2$-ary words of length $2^3-1$ which are $Z_3$
instances.  In Figure~\ref{fig:zinstree} we emphasize the recursive nature
of their construction by drawing the words as a binary tree with arrows
extending from a word to the words it generates.

\begin{figure}[!ht]
\centering
\begin{tikzpicture}[->,>=stealth', level/.style={sibling distance = 8cm/#1, level distance = 1.5cm}, scale=0.6,transform shape]
    \node [treenode] {$\epsilon$}
    child
    {%
        node [treenode] {$0$}
        child
        {%
            node [treenode] {$000$}
            child
            {%
                node [treenode] {$\phantom{000}0000000$}
            }
            child
            {%
                node [treenode] {$0001000\phantom{000}$}
            }
        }
        child
        {%
            node [treenode] {$010$}
            child
            {%
                node [treenode] {$\phantom{000}0100010$}
            }
            child
            {%
                node [treenode] {$0101010\phantom{000}$}
            }
        }
    }
    child
    {%
        node [treenode] {$1$}
        child
        {%
            node [treenode] {$101$}
            child
            {%
                node [treenode] {$\phantom{000}1010101$}
            }
            child
            {%
                node [treenode] {$1011101\phantom{000}$}
            }
        }
        child
        {%
            node [treenode] {$111$}
            child
            {%
                node [treenode] {$\phantom{000}1110111$}
            }
            child
            {%
                node [treenode] {$1111111\phantom{000}$}
            }
        }
    }
;
\end{tikzpicture}
\caption{}\label{fig:zinstree}
\end{figure}
\noindent
Now if we consider the lexical ordering of the $2$-ary words of length $3$,
we see (by inspection) that the lexical distances between the $Z_2$
instances are symmetric, except for the distance between $011$ and $100$.
That is, the lexical distance between $000$ and $010$ is equal to the
lexical distance between $101$ and $111$, and so on.  If we take the
ordered sequence of words of length $3$
\begin{equation*}
    000, 001, 010, 011, 100, 101, 110, 111
\end{equation*}
and transform it into a $2$-ary word such that the $i$-th symbol is $1$ if
and only if the $i$-th word in the sequence encounters $Z_2$, then we
construct the word $10100101$, which is a $Z_3$ instance.
Our main result is that words constructed in this way are always $Z_{n+1}$
instances.

\section{Definitions}
In this section we introduce some new definitions which we will use in the
statement and proof of our theorem.
In the remainder of the paper we use $L_q^m$ to denote the sequence of all
$q$-ary words of length $m$ over $\Sigma$ in lexical order,
and we use $\phi$ as an arbitrary homomorphism from $\mathbb{N}$ to
$\Sigma^+$.

\begin{definition}[Density]
    If a word $W$ encounters a word $V$ in $x$ different ways,
    we say that the \textit{density} of $V$ in $W$ is $x$.
    We write this as $\rho_V(W) = x$.

    \begin{example}
        For example if $W = 0110$ then
        \begin{align*}
            \rho_{0}(W) &= 10 & \rho_{01}(W) &= 4 & \rho_{010}(W) &= 1 \\
        \end{align*}
    \end{example}
    \noindent
    The maximum density of $Z_n$ in a $q$-ary word of length $m$ is
    denoted as $P(q, n, m)$ and calculated as the number of subwords of
    length at least $2^n-1$.  In symbols
    \begin{equation}
        \sum_{i=2^n-1}^m m - i = \frac{1}{2}(m-2^n+1)(m-2^n+2)
    \end{equation}
    When the context is clear, we will simply use $P$ without arguments.
\end{definition}

\begin{definition}[Index Function]
    If $W \in L_q^m$ then $\Delta_q^m(W)$
    is the index of $W$ in $L_q^m$.
    When $m, q$ are apparent from the context,
    we will write this as $\Delta(W)$.

    \begin{example}
        \begin{equation*}
            L_2^3 = 000, 001, 010, 011, 100, 101, 110, 111
        \end{equation*}
        \begin{align*}
            \Delta(000) &= 0 & \Delta(001) &= 1 & \Delta(010) &= 2 \\
            \Delta(011) &= 3 & \Delta(100) &= 4 & \Delta(101) &= 5 \\
            \Delta(110) &= 6 & \Delta(111) &= 7 \\
        \end{align*}
    \end{example}
\end{definition}

\begin{definition}[Truth Table]
    We define the \textit{truth table} $T_n(L_q^m) = t_1 \dots t_{q^m}$
    as the $2$-ary word of length $q^m$ such that $t_i = 1$
    if and only if the $i$-th word of $L_q^m$ is a $Z_n$ instance.

    \begin{example}
        \begin{align*}
            L_2^3      &= 000, 001, 010, 011, 100, 101, 110, 111 \\
            T_2(L_2^3) &= 10100101
        \end{align*}
    \end{example}
\end{definition}

\begin{definition}[Density Table]
    The \textit{density table} is a generalization of the truth table,
    where instead of indicating the presence of an encounter,
    we count the number of encounters.

    Let $W_i$ be the $i$-th word in $L_q^m$.
    We define the density table $D_n(L_q^m) = d_1 \dots d_{q^m}$ as the
    $P$-ary word of length $q^m$ such that $d_i$ = $\rho_{Z_n}(W_i)$.

    \begin{example}
        When $m = 2^n-1$, the truth table and density table are identical.
        \begin{align*}
            L_2^3      &= 000, 001, 010, 011, 100, 101, 110, 111 \\
            D_2(L_2^3) &= T_2(L_2^3) = 10100101.
        \end{align*}
        However, as $m$ increases the density table grows more complex
        \begin{align*}
            L_2^4      &=               0000, 0001, 0010, 0011, \\
                       &\phantom{{}={}} 0100, 0101, 0110, 0111, \\
                       &\phantom{{}={}} 1000, 1001, 1010, 1011, \\
                       &\phantom{{}={}} 1100, 1101, 1110, 1111  \\
            T_2(L_2^4) &= 1110111111110111 \\
            D_2(L_2^4) &= 3 1 2 0 2 3 1 1 1 1 3 2 0 2 1 3 \\ 
        \end{align*}

    \end{example}
    In Appendix 1 there are graphical depictions of density tables for small $n, m, q$.
\end{definition}

\begin{definition}[Neighbor]\label{zadj}
    If $U$ and $V$ are $Z_n$ instances of length $m$ such that $U \neq V$ and there does
    not exist a $Z_n$ instance $Y$ of length $m$ such that $\Delta(U) < \Delta(Y) < \Delta(V)$ or $\Delta(V) < \Delta(Y) < \Delta(U)$
    then $U$ and $V$ are said to be neighbors.
\end{definition}

\section{Properties of Truth Tables}\label{sec:ptt}
When we began this work, our goal was to investigate the structure of the density tables.
However, their structure turns out to be very difficult to describe.
As a first step towards their description, we will prove the following theorem.
\begin{theorem}\label{thm:zinst}
    $T_n(L_q^m)$ is a $2$-ary $Z_{n+1}$ instance if $m = 2^n-1$ and $1 < n$.
\end{theorem}
\noindent
Our proof will be by induction, and we will make use of three lemmas.
From Lemma~\ref{thm:zcnt} we will know the number of $Z_n$ instances of
length $2^n-1$,
from Lemma~\ref{thm:zdist} we will know how these instances are distributed
in $L_{q}^{m}$, and from Lemma~\ref{thm:zplus1} we will see how the structure
of the density table for words of length $m$ can be used to infer some
structure of the density table for words of length $m+1$.
We will then combine these results to prove the theorem.


\begin{lemma}\label{thm:zcnt}
    There are $q^n$ words of length $2^n-1$ which are $Z_n$ instances.
    \begin{proof}
        Consider the base case of $Z_1$.
        There are $q$ words of length $1$ which are $Z_1$ instances:
        $0, 1, \dots, Q$.

        Assume that this is true for all $n \leq k$.
        By definition, a $Z_{k+1}$ instance is of the form $\phi(Z_k) q_i \phi(Z_k)$ for some $\phi$ and some $q_i \in \Sigma$.
        By the induction hypothesis, the number of $Z_k$ instances of length $2^{k-1}-1$ is $q^{k-1}$.
        Therefore there are $q^{k-1}$ ways to chose $\phi(Z_{k})$,
        and clearly there are $q$ ways to choose $q_i$,
        which gives us $q(q^{k-1}) = q^{k}$ words of length $2^k-1$ which encounter $Z_k$.
    \end{proof}
\end{lemma}

\begin{definition}[Generate]
    If $U$ is a $Z_n$ instance of length $2^n-1$ we say the set of
    $Z_{n+1}$ instances of the form $U q_i U$ is \textit{generated} by $U$.
    From Lemma~\ref{thm:zcnt} we see that each instance generates a set of
    size $q$.
\end{definition}

\begin{lemma}\label{thm:zdist}
    Let $m=2^n-1$, and let $q_i$ range over $\Sigma$.
    If $U, V$ are neighboring $Z_n$ instances of length $m$ such that $U < V$
    then $UQU$ and $V0V$ are neighboring $Z_{n+1}$ instances, and
    \begin{align*}
        \Delta(U q_{i+1} U) - \Delta(U q_{i} U) &= m \\
        \Delta(V0V) - \Delta(UQU) &= q^{m+1}\Delta(V) + \Delta(V) - (q^{m+1}\Delta(U) + q^{m} + \Delta(U))
    \end{align*}

    \begin{proof}
        Consider the base case of $Z_1$.
        There are $q$ words of length $1$ which are $Z_1$ instances:
        $0, 1, \dots, Q$.
        If $U, V$ are two neighboring $Z_1$ instances such that $U < V$ then $\Delta(V) - \Delta(U) = 1$
        and by inspection we see that the two conditions hold
        \begin{align*}
            \Delta(U q_{i+1} U) - \Delta(U q_{i} U) &= m = 1 \\
            \Delta(V0V) - \Delta(UQU) &= q^2 \Delta(V) + \Delta(V) - (q^2\Delta(U) + q + \Delta(U))  \\
                                      &= q^2 (\Delta(V) - \Delta(U)) + (\Delta(V) - \Delta(U)) - q   \\
                                      &= q^2 - q + 1
        \end{align*}
        Assume that this is true for all $n \leq k$.
        Let $m=2^{k}-1, m^{\prime} = 2^{k+1}-1$, and let $U, V$ be
        neighboring $Z_{k}$ instances with $U < V$.
        As before, $U q_{i} U$ and $U q_{i+1} U$ are clearly neighboring
        $Z_{k+1}$ instances.
        There must therefore be $m$ words avoiding $Z_n$ since $\Delta(U q_{i} U) + q^m = \Delta(U q_{i+1} U)$.
        Furthermore, there can be no $Z_{k+1}$ instance of size $m^{\prime}$ which starts with any of the words between $U$ and $V$,
        since $U, V$ are neighbors and the $Z_{k+1}$ instances of length $m^{\prime}$ are necessarily of the form $\phi(Z_{k}) q_{i} \phi(Z_k)$.
        Therefore the index of $UQU$ in $L_q^{m^{\prime}}$ is $q^{m^{\prime}} \Delta(U) + q^{m} \Delta(Q) + \Delta(U)$,
        and the index of $V0V$ in $L_q^{m^{\prime}}$ is $q^{m^{\prime}} \Delta(V) + \Delta(V)$, and so the condition holds.
    \end{proof}

    \begin{example}
        The truth of this proof is most readily seen by viewing the words as natural numbers in base $q$ (with numerically meaningless leading zeros).
        Consider the following example for $q=2, n=2$ and $m=4$.
        \begin{alignat*}{4}
            &\Delta(010)     &&= 2 \\
            &\Delta(0100000) &&= 2^{5}             &&= 2^{4}\Delta(010) \\
            &\Delta(0100010) &&= 2^{5} + 2         &&= 2^{4}\Delta(010) + \Delta(010) \\
            &\Delta(0101010) &&= 2^{5} + 2^{3} + 2 &&= 2^{4}\Delta(010) + 2^{3} + \Delta(010) \\
        \end{alignat*}
    \end{example}
\end{lemma}

\begin{definition}
    The arithmetic in the following proofs is greatly simplified by making
    use of a utility function $\Psi_n$ which encodes the second condition
    of Lemma~\ref{thm:zdist}.

    Let $U,V$ be neighboring $Z_{n-1}$ instances such that $V$ is the
    $i$-th $Z_{n-1}$ instance in $L_q^{2^{n-1}-1}$ and $U<V$.
    \begin{equation*}
        \Psi_n(i) = \Delta(V0V) - \Delta(UQU)
    \end{equation*}
    \begin{example}
        \begin{align*}
            n &= 2 \\
            q &= 3 \\
            U &= 010 \\
            V &= 020 \\
            \Psi_n(1) &= \Delta(0200020) - \Delta(0102010)
        \end{align*}
    \end{example}
\end{definition}

\begin{lemma}\label{thm:zplus1}
    Let $D_n = D_n(L_q^m) = d_0 \dots d_{q^m}$ and
    $D_n' = D_n(L_q^{m+1}) = d_0' \dots d_{q^{m+1}}'$.
    For all natural numbers $i, r, x$ with $i < q^m, r \leq m$ we have
    \begin{align*}
        d_i = x \implies d_i \leq d_{i q^r}' \\
    \end{align*}
\end{lemma}
\begin{proof}
    The proof of this is trivial, and follows from the facts that when a
    word $U$ encounters a Zimin word, so does $UV$ and $VU$ for any word
    $V$.
\end{proof}

We will now formally state and prove our main result.
\begin{theorem}\label{thm:zinst2}
    The truth table $T_n(L_q^m)$ is a $Z_{n+1}$ instance if $1<n$ and $2^n-1=m$.
    Specifically, ${\phi_n(Z_{n+1}) = T_n(L_q^m)}$ where
    \begin{equation}
        \phi_n(i) =
        \begin{cases}
            1{(0^{q^{2^{n-1}-1}}1)}^{q-1}      & i = 0 \\
            0^{\Psi_n(i)}                      & 0 < i < n \\
        \end{cases}
    \end{equation}
\end{theorem}
\begin{proof}
    Consider the base case of $n = 3$.
    It is easy to verify the direct application of Lemmas~\ref{thm:zcnt}
    and Lemma~\ref{thm:zdist} that
    $\phi_3(Z_3) = T_2 = 1{(0^{q-1}1)}^{q-1}0^{q}1{(0^{q-1}1)}^{q-1}$,
    and that the sequence is a $Z_3$ instance.
    Assume that this holds for all $n < k$.
    For ease of notation, let $m = 2^{k-1}-1, m' = 2^{k}-1$,
    and ${T = T_{k-1}(L_q^{m}) = t_0 \dots t_{q^{m}}}$,
    ${T' = T_{k}(L_q^{m'}) = t_0' \dots t_{q^{m'}}'}$.


    We must show that $\phi_{k}(Z_{k+1}) = T'$.
    By Lemma~\ref{thm:zcnt} there are $q^{k-1}$ instances of $Z_{k-1}$ in $L_q^m$,
    and each $Z_{k-1}$ instance generates $q$ instances of $Z_{k}$ in $L_q^{m'}$,
    each of the form $\phi(Z_{k-1})q_i\phi(Z_{k-1})$.
    Let $U$ be such a $Z_{k-1}$ instance.
    By the first condition of Lemma~\ref{thm:zdist},
    $\Delta(U q_{i+1} U) - \Delta(U q_i U) = q^{i+1}-q^{i} = q^{\len{U}} = q^{2^{k-1}-1}$,
    from which it follows that for every pair of neighboring instances in $L_q^{m'}$ there are
    a pair of ``neighboring'' occurrences of the symbol $1$ in $T_{k}$ with $q^m$
    occurrences of the symbol $0$ between them.
    Furthermore, $\phi_{k}^{-1}(\Delta^{-1}(U q_i U)) = 0$,
    since $\phi_{k-1}(\Delta^{-1}(U)) = 0$ by the inductive hypothesis.

    It remains to show that the number of occurrences of the symbol $0$
    separating the occurrences of the word $\phi_{k}(0)$ in $T'$ is
    given by $\phi_{k}(i)$ for $0 < i \leq k$.
    Let $U, V$ be neighboring $Z_{k-1}$ instances such that $U < V$.
    It follows that $UQU, V0V$ are neighboring $Z_{k}$ instances.
    Specifically, $UQU$ corresponds to the symbol $1$ with the greatest
    index in the subword of $T_k$ corresponding to the $q$ words generated
    by $U$, and $V0V$ corresponds to the least index of the symbol $1$ in
    the subword of $T_k$ corresponding to the $q$ words generated by $V$.
    Recall that we denote $\Delta(V0V) - \Delta(UQU)$ as the image of
    $\Delta^{-1}(V)$ under $\Psi_k$.
    It follows that the number of occurrences of the symbol $0$ in $T_{k}$
    between the subword generated by $U$ and the subword generated by $V$
    is the image of $\Delta^{-1}(V)$ under $\Psi_k$.
    From this, it follows that $\phi_{k}(Z_{k+1}) = T'$.
\end{proof}

\section{Continuing Work and Connections to Number Theory}
The proof of Theorem~\ref{thm:zinst2} can be extended to show that
$T(L_q^{m})$ is a $Z_{n+2}$ instance when $m=2^n$.
However, this pattern does not generally hold for larger $m$,
instead the structure of the truth table becomes more and more complicated
as $m$ continues to grow, until suddenly ``collapsing'' into a word consisting
only of the symbol $1$.

In Appendix 1 we give graphical depictions of the density tables for small
$n, m, q$.
A glance at the images is sufficient to see that this result just barely
scratches the surface of the density table structure.
We believe that better understanding the structure of the density tables would
help a great deal in improving the bounds on Zimin word avoidance,
as introduced in~\cite{cooper2014} and independently in~\cite{rytter2015}.
Furthermore, we believe that this structure would best be explored through a
number theoretic framework.
Our work in this area has just begun,
but a few examples of our findings are sketched below.

\subsection{Connection To Integer Partitions}
\begin{theorem}
    For all natural numbers $r$ such that $q^{f(n,q)} \leq r$, there exists
    natural numbers $u,v,x,y$ such that some partition of $r$ has at most
    $5$ parts such that
    \begin{equation*}
        r = xq^{\log(v+2u+y+1)} +
            uq^{\log(v+u+y+1)} +
            vq^{\log(u + y)} +
            uq^{\log(y)} + y
    \end{equation*}
    where $v, x, y$ may be $0$, all logarithms are taken base $q$ and
    $\Delta^{-1}(u)$ is a $Z_{n-1}$ instance.
    (And where we abuse notation slightly to define $\log(0) = 0$.)
\end{theorem}

\begin{proof}
    This is simply a restatement of the existence of Zimin words in the
    language of integer partitions, which uses the definition of the $\Delta$
    function to convert between words and numbers.
\end{proof}

\subsection{Connection To Prime Numbers}
We conjecture that there is a $q$-ary word $W$ of length $f(n, q)-1$ such that
$\Delta(W)$ is prime.
There are several reasons why this conjecture seems likely,
but for now we will state only this one small observation.

Consider the truth table $T_2(L_2^4) = 1110111111110111$.
We see that the fourth symbol (counting from $1$) is $0$.
If we consider ${T_2(L_2^4)}^k$ then we see that the symbol at index $k 2^4 + 3$
must always be $0$ when $k$ is a positive integer.
By Dirichlet's theorem this sequence contains an infinite number of primes,
since $2^4$ and $3$ are coprime.
Combined with Lemma~\ref{thm:zplus1} this is enough to see that there are an
infinite number of prime candidates for avoiding $Z_n$.

\medskip

\bibliographystyle{abbrv}
\bibliography{truth.table.bib}

\newpage

\begin{appendices}
    \section{Density Tables}
    The following graphics depict the density tables for a few small values of $n, m, q$.
    Note that some of the tables for $2 < q$ look asymmetric,
    but this is due to artifacts resulting from scaling the images.
    \begin{figure}[!ht]
      \centering
      \includegraphics[width=\textwidth]{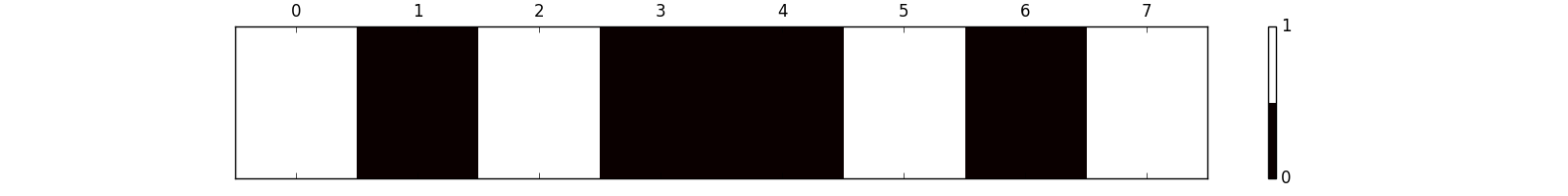}
      \caption{$n=2, m=3, q=2$}
    \end{figure}
    \begin{figure}[!ht]
      \centering
      \includegraphics[width=\textwidth]{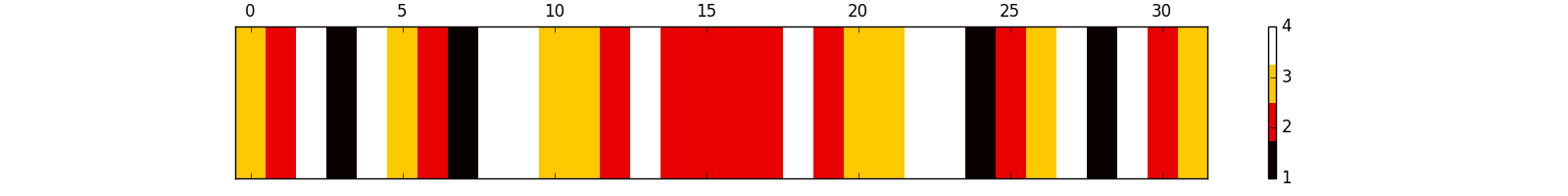}
      \caption{$n=2, m=5, q=2$}
    \end{figure}
    \begin{figure}[!ht]
      \centering
      \includegraphics[width=\textwidth]{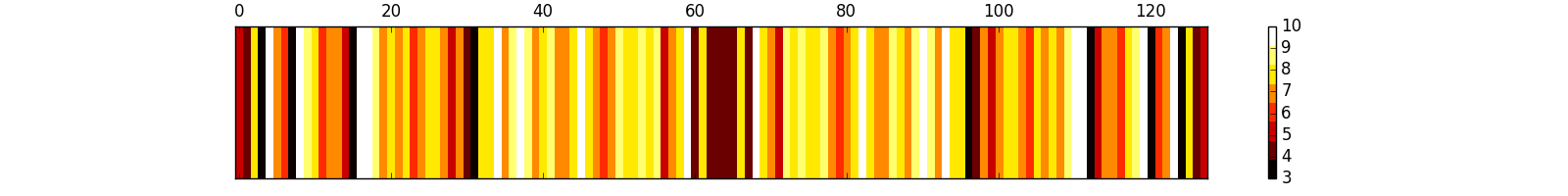}
      \caption{$n=2, m=7, q=2$}
    \end{figure}
    \begin{figure}[!ht]
      \centering
      \includegraphics[width=\textwidth]{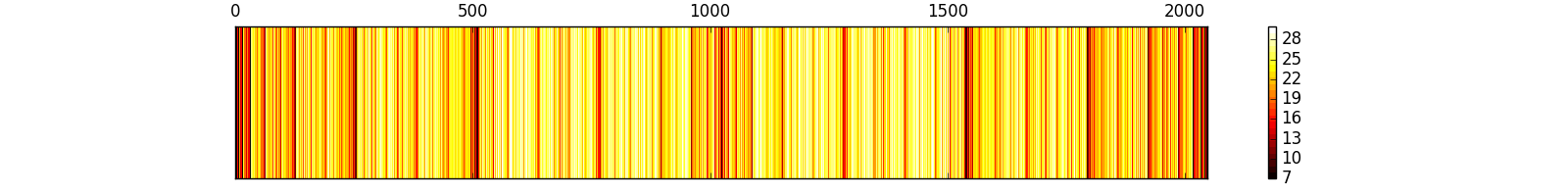}
      \caption{$n=2, m=11, q=2$}
    \end{figure}
    \begin{figure}[!ht]
      \centering
      \includegraphics[width=\textwidth]{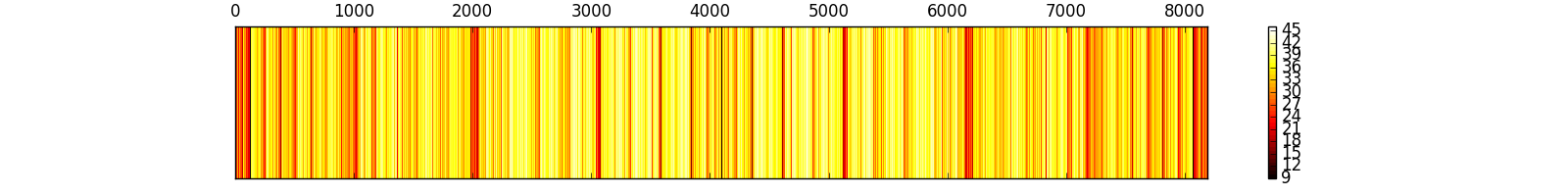}
      \caption{$n=2, m=13, q=2$}
    \end{figure}

    \clearpage

    \begin{figure}[!ht]
      \centering
      \includegraphics[width=\textwidth]{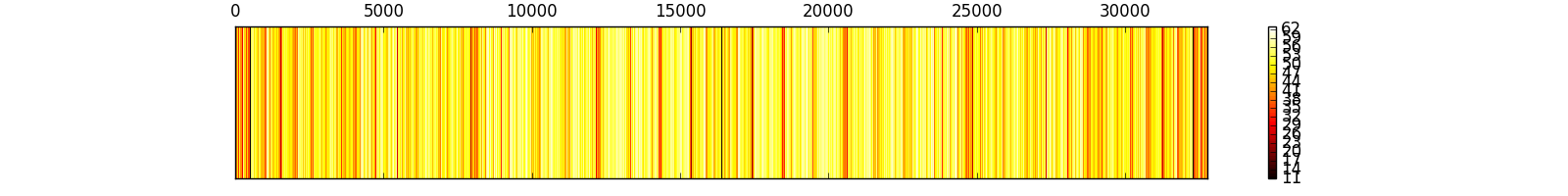}
      \caption{$n=2, m=15, q=2$}
    \end{figure}
    \begin{figure}[!ht]
      \centering
      \includegraphics[width=\textwidth]{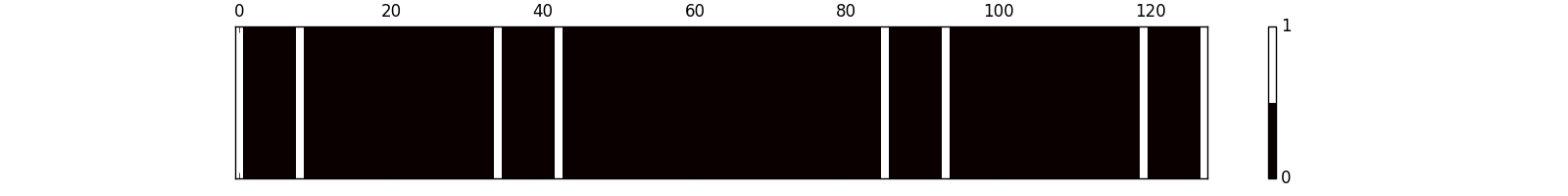}
      \caption{$n=3, m=7, q=2$}
    \end{figure}
    \begin{figure}[!ht]
      \centering
      \includegraphics[width=\textwidth]{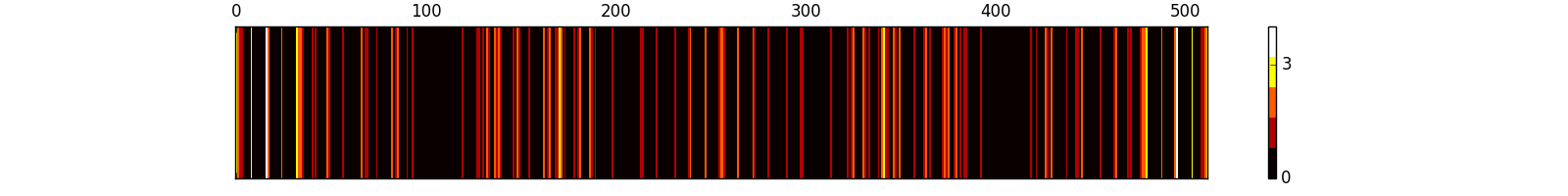}
      \caption{$n=3, m=9, q=2$}
    \end{figure}
    \begin{figure}[!ht]
      \centering
      \includegraphics[width=\textwidth]{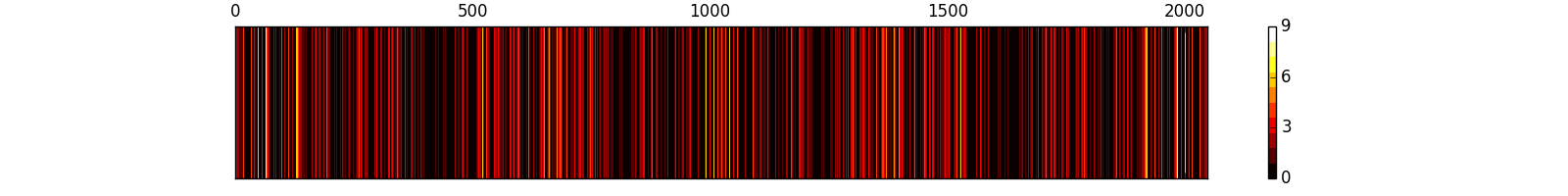}
      \caption{$n=3, m=11, q=2$}
    \end{figure}
    \begin{figure}[!ht]
      \centering
      \includegraphics[width=\textwidth]{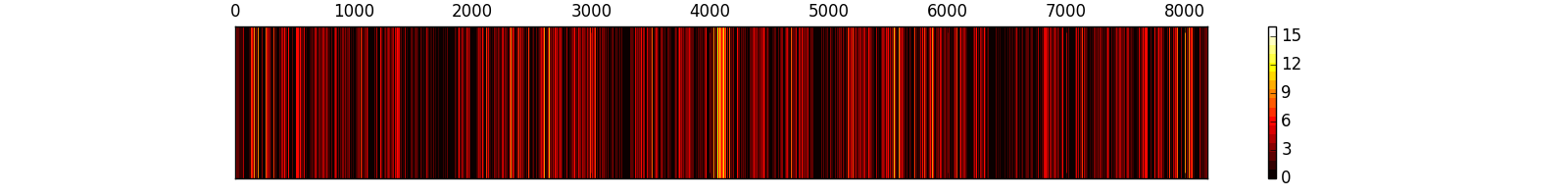}
      \caption{$n=3, m=13, q=2$}
    \end{figure}
    \begin{figure}[!ht]
      \centering
      \includegraphics[width=\textwidth]{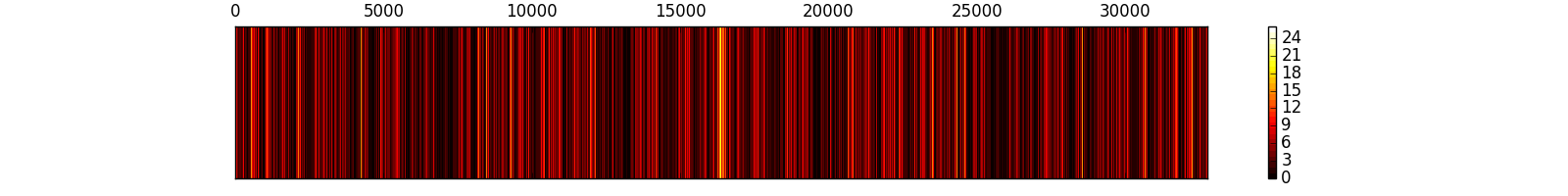}
      \caption{$n=3, m=15, q=2$}
    \end{figure}
    \begin{figure}[!ht]
      \centering
      \includegraphics[width=\textwidth]{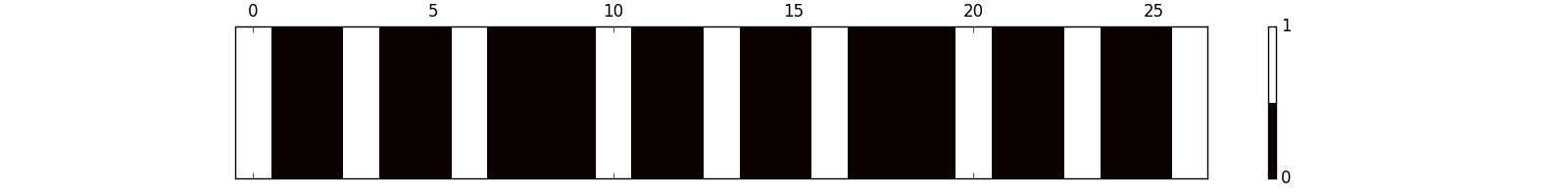}
      \caption{$n=2, m=3, q=3$}
    \end{figure}
    \begin{figure}[!ht]
      \centering
      \includegraphics[width=\textwidth]{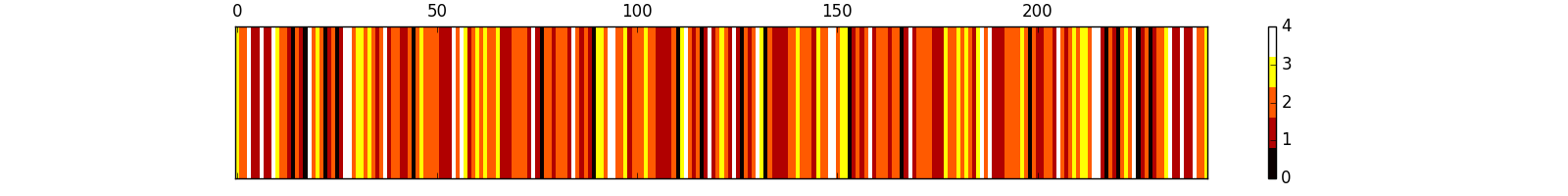}
      \caption{$n=2, m=5, q=3$}
    \end{figure}

    \clearpage

    \begin{figure}[!ht]
      \centering
      \includegraphics[width=\textwidth]{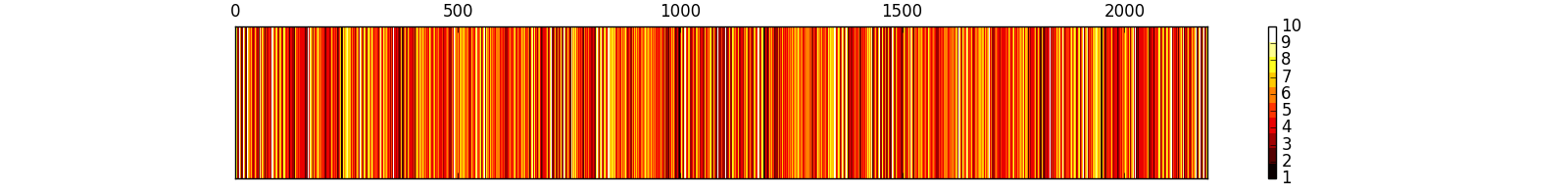}
      \caption{$n=2, m=7, q=3$}
    \end{figure}
    \begin{figure}[!ht]
      \centering
      \includegraphics[width=\textwidth]{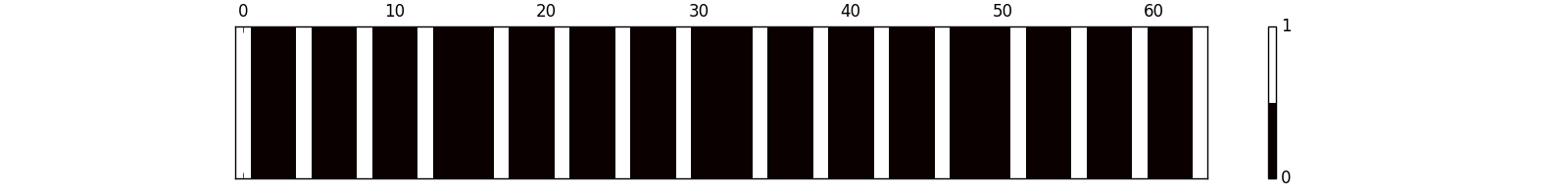}
      \caption{$n=2, m=3, q=4$}
    \end{figure}
    \begin{figure}[!ht]
      \centering
      \includegraphics[width=\textwidth]{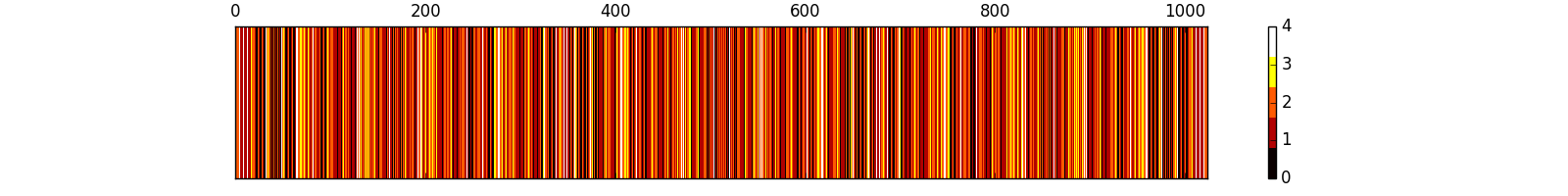}
      \caption{$n=2, m=5, q=4$}
    \end{figure}
    \begin{figure}[!ht]
      \centering
      \includegraphics[width=\textwidth]{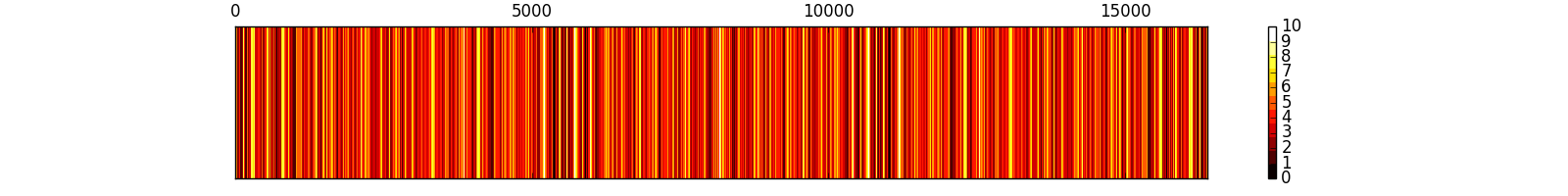}
      \caption{$n=2, m=7, q=4$}
    \end{figure}
    \begin{figure}[!ht]
      \centering
      \includegraphics[width=\textwidth]{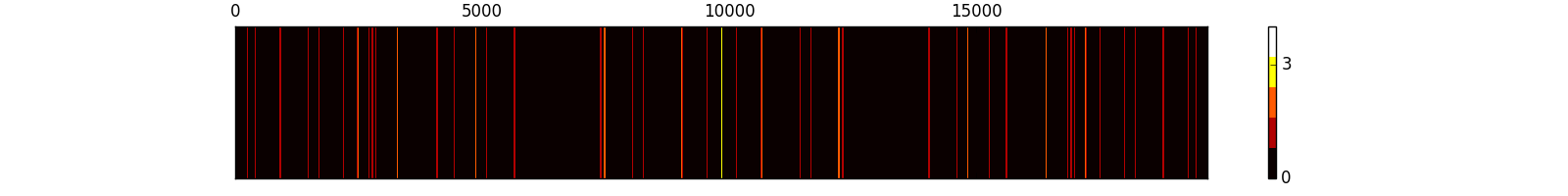}
      \caption{$n=3, m=9, q=3$}
    \end{figure}
    \begin{figure}[!ht]
      \centering
      \includegraphics[width=\textwidth]{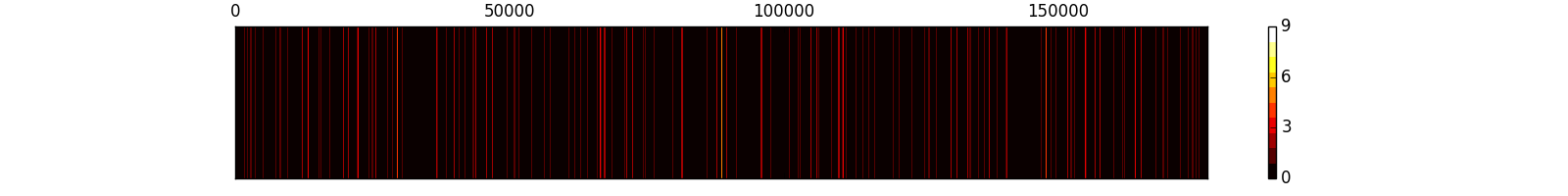}
      \caption{$n=3, m=11, q=3$}
    \end{figure}
\end{appendices}

\end{document}